\documentclass[article,12pt]{amsart}
\usepackage{amsfonts}
\usepackage{amsmath}
\usepackage{graphicx}
\usepackage{color}
\usepackage{float}
\usepackage{epsfig}

\newcommand{\dR}{\mathbb{R}}

\newcommand{\cA}{\mathcal{A}}
\newcommand{\cB}{\mathcal{B}}
\newcommand{\cC}{\mathcal{C}}

\newcommand{\cH}{\mathcal{H}}
\newcommand{\cN}{\mathcal{N}}

\newcommand{\cT}{\mathcal{T}}

\newcommand{\cZ}{\mathcal{Z}}
\newcommand{\cF}{\mathcal{F}}

\newcommand{\rI}{\mathrm{I}}
\newcommand{\rJ}{\mathrm{J}}
\newcommand{\cM}{<\!\!M\!\!>}

\newcommand{\veps}{\varepsilon}
\newcommand{\wh}{\widehat}

\newcommand{\ind}{\mbox{1}\kern-.25em \mbox{I}}
\font\calcal=cmsy10 scaled\magstep1
\def\build#1_#2^#3{\mathrel{\mathop{\kern 0pt#1}\limits_{#2}^{#3}}}
\def\liml{\build{\longrightarrow}_{}^{{\mbox{\calcal L}}}}

\def\videbox{\mathbin{\vbox{\hrule\hbox{\vrule height1ex \kern.5em
\vrule height1ex}\hrule}}}
\def\demend{\hfill $\videbox$\\}

\numberwithin{equation}{section}
\theoremstyle{plain}
\newtheorem{thm}{Theorem}[section]
\newtheorem{rem}{Remark}[section]
\newtheorem{cor}{Corollary}[section]
\newtheorem{lem}{Lemma}[section]

\ifx\undefined\ly 
\gdef\beginguillemets{\leavevmode\raise0.3ex%
         \hbox{{$\scriptscriptstyle\langle\!\langle\,$}\nobreak\ignorespaces}} 
\else 
\gdef\beginguillemets{\leavevmode\hbox{\ly(\kern-0.20em(\kern+0.20em}\nobreak} 
\fi 
\ifx\undefined\ly 
\gdef\endguillemets{\ifdim\lastskip>\z@\unskip\penalty\@M\fi 
         \leavevmode\raise0.3ex%
         \hbox{{$\scriptscriptstyle\,\rangle\!\rangle$}}} 
\else 
\gdef\endguillemets{\nobreak\leavevmode\hbox{\kern+0.20em\ly)\kern-0.20em)}} 
\fi 
\newcommand{\og}{{\scriptscriptstyle\langle\!\langle\,}} 
\newcommand{\fg}{{\scriptscriptstyle\,\rangle\!\rangle}}


\begin{document}
\title[A Durbin-Watson serial correlation test for ARX processes]
{A Durbin-Watson serial correlation test for ARX processes via excited adaptive tracking}
\author{Bernard Bercu}
	\address{Universit\'e de Bordeaux, Institut de Math\'ematiques de Bordeaux,
	UMR 5251, 351 cours de la lib\'eration, 33405 Talence cedex, France.}

\author{Bruno Portier}
\address{Normandie Universit\'e, D\'epartement de G\'enie Math\'ematiques, Laboratoire de Math\'ematiques, 
INSA de Rouen, LMI-EA 3226, place Emile Blondel, BP 08, 76131 Mont-Saint-Aignan cedex, France}

\author{Victor Vazquez}
\address{Universidad Aut\'onoma de Puebla, Facultad de Ciencias F\'isico
Matem\'aticas, Avenida San Claudio y Rio Verde, 72570 Puebla, Mexico.}

\begin{abstract}
We propose a new statistical test for the residual autocorrelation in ARX adaptive tracking. 
The introduction of a persistent excitation in the adaptive tracking control allows us
to build a bilateral statistical test based on the well-known Durbin-Watson statistic.
We establish the almost sure convergence and the asymptotic normality for the 
Durbin-Watson statistic leading to a powerful serial correlation test.
Numerical experiments illustrate the good performances of our statistical test procedure.
\end{abstract}

\maketitle

\vspace{-1cm}

\section{Introduction}
\label{SectionIM}
\vspace{1ex}

Model validation is an important and essential final step in the identification of stochastic dynamical systems.
This validation step is often done through the analysis of residuals of the model considered.
In particular, testing the non-correlation of the residuals is a crucial task since many theoretical results 
require independence of the driven noise of the systems. Moreover, non-compliance with this hypothesis
can lead to misinterpretation of the theoretical results.
For example, it is well known that for linear autoregressive models with autocorrelated residuals,
the least squares estimator is asymptotically biased, see e.g. \cite{BPDW2012}, \cite{D1970}, \cite{N1966},
\cite{Proia2013}, and therefore the estimated model is not the correct one.
Consequently, to ensure a good interpretation of the results,
it is necessary to have a powerful tool allowing to detect the possible autocorrelation
of the residuals.
The well-known statistical test of Durbin-Watson was introduced to deal with this question, and more specifically, 
for detecting the presence of a first-order autocorrelated noise in linear regression models
\cite{DW1}, \cite{DW2}, \cite{DW3}, \cite{D1970}, firstly
and for linear autoregressive models \cite{BPDW2012}, \cite{N1966}, 
\cite{SRI1987}, \cite{STO2007}, \cite{Proia2013}, secondly.
\vspace{1ex}

To the best of our knowledge, no such serial correlation statistical test 
is available for controlled autoregressive processes.
The aim of this paper is to carry out a serial correlation test, based on the Durbin-Watson statistic, for the ARX$(p,1)$ process given, for all $n\geq 0$, by
\begin{equation}\label{ARXPintro}
X_{n+1}=\sum_{k=1}^p\theta_k X_{n-k+1}+U_n+\varepsilon_{n+1}
\end{equation}
where the driven noise $(\varepsilon_{n})$ is given by the first-order autoregressive process
\begin{equation}\label{CORintro}
\varepsilon_{n+1}= \rho\,\varepsilon_{n}+V_{n+1}
\end{equation}
and the control objective is the tracking of a given reference trajectory.
More precisely, we shall propose a bilateral statistical test allowing to decide between the null hypothesis ${\mathcal H}_0:\og \rho = 0\fg$ 
which ensures that the driven noise is not correlated,
and the alternative one
${\mathcal H}_1:\og \rho \neq 0\fg$
which means that the residual process is effectively first-order autocorrelated.
The choice of the Durbin-Watson statistic, instead of any other statistical tests, 
is governed by its efficienty for autoregressive processes without control, see \cite{BPDW2012}, \cite{D1970}, \cite{N1966}, \cite{Proia2013}.
\ \vspace{1ex} \par
In contrast with the recent work  \cite{BPV2012}, we propose to make use of a different strategy via a modification of the adaptive control law. This modification relies on the introduction
of an additional persistent excitation. Since the pioneering works of Anderson \cite{AND1982} and Moore \cite{M1983}, the concept of persistent excitation has been successfully developped 
in many fields of applied mathematics such as identification of complex systems,
feedback adaptive control, etc. 
While it was not possible in  \cite{BPV2012} to test the non correlation 
of the driven noise $(\varepsilon_{n})$, that is to test whether or not $\rho = 0$, 
the introduction of an additional persistent excitation term in the control law 
will be the key point to build our serial correlation test.
Moreover, we wish to mention that all previous works devoted to non correlation test based 
on the Durbin-Watson statistic were only related to uncontrolled processes.
Therefore, thanks to the persistent excitation, our statistical test is at our knowledge 
the first one in the context of linear processes with adaptive control.
\vspace{1ex}

The paper is organized as follows. Section 2 is devoted to the ARX process and to the persistently excited adaptive control law.
In Section 3, we establish the asymptotic properties of the Durbin-Watson statistic as well as a bilateral 
statistical test for residual autocorrelation. Some numerical experiments are provided in Section 4. 
Finally, all technical proofs are postponed in the Appendices.

\section{Model and excited adaptive tracking}
\label{SectionEC}

We focus our attention on the $\text{ARX}(p,1)$ process given, for all $n\geq 0$, by
\begin{equation} \label{ARXP}
X_{n+1}=\sum_{k=1}^p\theta_k X_{n-k+1}+U_n+\varepsilon_{n+1}
\end{equation}
where the driven noise $(\varepsilon_{n})$ is given by the first-order autoregressive process
\begin{equation} \label{COR}
\varepsilon_{n+1}= \rho\,\varepsilon_{n}+V_{n+1}.
\end{equation}
We assume that the autocorrelation parameter satisfies $|\rho|<1$ and the initial values $X_0$, 
$\varepsilon_0$ and $U_0$ may be arbitrarily chosen. We also assume that $(V_n)$ is a martingale difference
sequence adapted to the filtration $\mathbb{F} = (\mathcal{F}_n)$ where $\mathcal{F}_n$ is
the $\sigma$-algebra of the events occurring up to time $n$, such that, for all $n \geq 0$, 
$\mathbb{E}\left[V_{n+1}^2|\mathcal{F}_n\right]=\sigma^2$ a.s. 
with $\sigma^2>0$.  We denote by $\theta$ the unknown parameter of the $\text{ARX}(p,1)$ process,
$$\theta^{t}  = (\theta_1,  \ldots, \theta_p).$$

Our control strategy is to regulate the dynamic of the process 
$(X_n)$ by forcing $X_n$ to track a bounded reference
trajectory $(x_n)$. We assume that $(x_n)$ is predictable which means that
for all $n\geq 1$, $x_n$ is $\cF_{n-1}$-measurable.
For the sake of simplicity, we also assume that
\begin{equation}  \label{CT}
\sum_{k=1}^{n} x_{k}^{2} =o(n) \hspace{1cm} \text{a.s.}
\end{equation}

In order to regulate the dynamic of the process $(X_n)$ given by (\ref{ARXP}), we propose to
make use of the adaptive control law introduced in \cite{BPV2012} together with additional persistent excitation. The strategy consists of using a control associated with a higher 
order model than the initial $\text{ARX}(p,1)$, and more precisely an $\text{ARX}(p+1,2)$ model.
The introduction of an additional excitation in the control law
will be the key point to build our serial correlation test 
for the driven noise  $(\varepsilon_{n})$, that is to test wether
or not $\rho = 0$. Denote by $(\xi_n)$ a centered exogenous noise
with known variance $\nu^2 > 0$, which will play the role of the additional excitation.
We assume that $(\xi_n)$ is independent of $(V_n)$, of $(x_n)$ and of the initial state of the system.
One can observe that these assumptions are not at all restrictive as we have in our own hands
the additional excitation $(\xi_n)$
\vspace{1ex}

The excited adaptive control law is given, for all $n \geq 0$, by
\begin{equation} \label{CONTROL}
U_n = x_{n+1}-\wh{\vartheta}_n^{\,t}\,\Phi_n+\xi_{n+1}
\end{equation}
where $\wh{\vartheta}_n$ stands for the least squares estimator of the unknown parameter
of the $\text{ARX}(p+1,2)$ model with uncorrelated driven noise
\begin{equation}
\label{EQFONDA}
X_{n+1}=\vartheta^t\Phi_n+U_n+V_{n+1}
\end{equation}
where the new parameter $\vartheta \in \dR^{p+2}$ is related to $\theta$ and $\rho$ by the identity
\begin{equation}
\label{NEWPARA}
\vartheta = 
\begin{pmatrix}
\  \theta \ \\
\ 0 \ \\
0
\end{pmatrix}
-
\rho
\begin{pmatrix}
 -1 \ \\
\ \, \theta \ \\
\ 1
\end{pmatrix}
\end{equation}
and the new regression vector is given by
$$\Phi_n^{t} = (X_n, \ldots, X_{n-p}, U_{n-1}).$$ 
It is well-known that $\wh{\vartheta}_n$ satisfies the recursive relation
\begin{equation}  
\label{LSVARTHETA}
\wh{\vartheta}_{n+1}=
\wh{\vartheta}_{n}+S_{n}^{-1}\Phi_{n}
\Bigl(X_{n+1}-U_{n}-\wh{\vartheta}_{n}^{\, t}\Phi_n\Bigr)
\end{equation}
where the initial value $\wh{\vartheta}_{0}$ may be arbitrarily chosen and 
\begin{equation}\label{DEFSN}
S_{n}=\sum_{k=0}^{n}\Phi_{k}\Phi_{k}^{t}+\rI_{p+2}.
\end{equation}
As usual, the identity matrix $\rI_{p+2}$ is added in order
to avoid useless invertibility assumption. 
One can immediately see from \eqref{NEWPARA}  that the last component of
the vector $\vartheta$ is $-\rho$. 
Consequently, we obtain an estimator of $\rho$ by simply taking the opposite of the last coordinate of  $\wh{\vartheta}_n$
which will be denoted by  $\wh{\rho}_n$.
In addition, one can also deduce from \eqref{NEWPARA} that
\begin{equation}\label{INIPARA}
\begin{pmatrix}\ \theta \ \\ \ \rho \ \end{pmatrix}
 = \Delta\, \vartheta
\end{equation}
where $\Delta$ is the rectangular matrix of size $(p+1)\!\times\!(p+2)$ given by
\begin{equation}
\label{DEFDELTA}
\Delta=
\begin{pmatrix}
1 & 0 & \cdots & \cdots & \cdots & 0 & 1 \\ 
\rho & 1 & 0 & \cdots & \cdots & 0 & \rho \\ 
\rho^2 & \rho & 1 & 0 & \cdots & 0 & \rho^2 \\ 
\cdots & \cdots & \cdots & \cdots & \cdots & \cdots & \cdots \\ 
\rho^{p-1} & \rho^{p-2} & \cdots & \rho & 1 & 0 & \rho^{p-1} \\ 
0 & 0 & \cdots & \cdots & \cdots & 0 & -1
\end{pmatrix}.
\end{equation}
Then, starting from (\ref{INIPARA}) and replacing $\rho$ by $\wh{\rho}_n$ in \eqref{DEFDELTA},
we can estimate $\theta$ by
\begin{equation}
\label{THETACHAP}
\wh{\theta}_n\ = \
 \Bigl(\rI_{p}\ \ 0 \Bigr) \wh{\Delta}_{n} \, \wh{\vartheta}_n
\end{equation}
where $\wh{\vartheta}_n$ is given by \eqref{LSVARTHETA} and
\begin{equation}
\label{DEFDELTAN}
\wh{\Delta}_{n}=
\begin{pmatrix}
1 & 0 & \cdots & \cdots & \cdots & 0 & 1 \\ 
\wh{\rho}_{n} & 1 & 0 & \cdots & \cdots & 0 & \wh{\rho}_{n} \\ 
\wh{\rho}_{n}^{\,2} & \wh{\rho}_{n} & 1 & 0 & \cdots & 0 & \wh{\rho}_{n}^{\,2}  \\ 
\cdots & \cdots & \cdots & \cdots & \cdots & \cdots & \cdots \\ 
\wh{\rho}_{n}^{\,p-1}  & \wh{\rho}_{n}^{\,p-2} & \cdots & \wh{\rho}_{n} & 1 & 0 & \wh{\rho}_{n}^{\,p-1}  \\ 
0 & 0 & \cdots & \cdots & \cdots & 0 & -1
\end{pmatrix}
.
\end{equation}
From the almost sure convergence of $\wh{\vartheta}_n$ to $\vartheta$, we easily deduce the almost sure convergences of
$\wh{\theta}_n$ and $\wh{\rho}_n$ to $\theta$ and $\rho$, respectively.

\section{A Durbin-Watson serial correlation test}
\def\Hzero{{\mathcal H}_0} 
\def\Hone{{\mathcal H}_1} 

We are in the position to introduce our serial correlation test based on
the Durbin-Watson statistic which is certainly the most commonly used statistics for
testing the presence of serial autocorrelation. Our goal is to test
$$\Hzero\,:\,\og \rho = 0\fg \hspace{1cm} \mbox{\rm vs} \hspace{1cm} \Hone\,:\,\og \rho \neq 0\fg.$$ 
For that purpose, we consider the Durbin-Watson statistic \cite{BPDW2012}, \cite{DW1}, \cite{DW2}, \cite{DW3}, \cite{D1970} given, for all $n\geq 1$, by 
\begin{equation} \label{DEFDN}
\widehat{D}_n=\frac{\sum_{k=1}^{n}
\left(\widehat{\varepsilon}_k-\widehat{\varepsilon}_{k-1}\right)^2}{\sum_{k=0}^{n}\widehat{\varepsilon}_{k}^{\,2}}
\end{equation}
where the residuals $\widehat{\varepsilon}_{k}$ are defined, for all $0\leq k\leq n$, by
\begin{equation}\label{RESIDUALN}
\widehat{\varepsilon}_{k}=X_k-U_{k-1}-\wh{\theta}_n^{\,t}\varphi_{k-1}
\end{equation}
with $\wh{\theta}_{n}$ given by \eqref{THETACHAP} and $\varphi_n^t = (X_n, \ldots, X_{n-p+1})$.
The initial value $\widehat{\varepsilon}_{0}$ may be arbitrarily chosen and we take $\widehat{\varepsilon}_{0}=X_{0}$. 
\ \vspace{1ex} \par
On the one hand, we would like to emphasize that it is not possible
to perform this statistical test if the control law is not persistently excited \cite{BPV2012}. On the other hand, one can notice
that it is also possible to estimate the serial correlation parameter $\rho$ 
by the least squares estimator
\begin{equation} \label{NEWLSRHO}
\overline{\rho}_n=\frac{\sum_{k=1}^{n}\widehat{\varepsilon}_k\widehat{\varepsilon}_{k-1}}{\sum_{k=1}^{n}\widehat{\varepsilon}_{k-1}^{\,2}}
\end{equation}
which is certainly the more natural estimator of $\rho$.
The Durbin-Watson statistic $\wh{D}_n$ is related to $\overline{\rho}_n$ by the linear relation
\begin{equation}
\label{RELRHODN}
\wh{D}_n=2(1-\overline{\rho}_n)+\zeta_n
\end{equation}
where the remainder $\zeta_n$ plays a negligeable role. 
The almost sure properties of $\wh{D}_n$ and $\rho_n$ are as follows.

\begin{thm}
\label{T-ASCVGDW}
Assume $(V_n)$ has a finite conditional moment of order $>2$. Then, $\overline{\rho}_n$ converges almost surely to $\rho$
\begin{equation}  
\label{ASCVGNRHO}
( \overline{\rho}_{n}-\rho)^{2}= \mathcal{O} \left( \frac{\log n}{n} \right) 
\hspace{0.5cm}\text{a.s.}
\end{equation}
In addition, $\wh{D}_{n}$ converges almost surely to $D=2(1-\rho)$.
Moreover, if $(V_n)$ has a finite conditional moment of order $>4$, we also have
\begin{equation}  
\label{ASCVGDW} 
\Bigl( \wh{D}_{n}-D \Bigr)^{2}= \mathcal{O} \left( \frac{\log n}{n} \right) 
\hspace{0.5cm}\text{a.s.}
\end{equation}
\end{thm}
\noindent{\bf Proof.}
The proofs are given in Appendix\,A. \demend
\ \vspace{1ex} \\
Let us now give the asymptotic normality of the Durbin-Watson statistic
which will be useful to build our serial correlation test.

\begin{thm}
\label{T-CLTDW}
Assume that $(V_n)$ has finite conditional moments of order $>2$. 
Then, we have
\begin{equation}
\label{CLTRHOO}
\sqrt{n}(\overline{\rho}_{n}-\rho)\liml
\cN\left(0,\tau^2\right)
\end{equation}
where the asymptotic variance $\tau^2$ is given by
\begin{eqnarray}\label{EXPTAU}
\tau^2 &=& \frac{(1-\rho^2)}{(\sigma^2+\nu^2)(\nu^2+\sigma^2\rho^{2(p+1)})}
\bigg[ 
\Bigl((\sigma^2 - \nu^2) - (p+1)\sigma^2\rho^{2p} +
(p-1)\sigma^2\rho^{2(p+1)}\Bigr)^2  \notag \\
& & \hspace{1cm} +\ \ 
\sigma^2(\nu^2+\sigma^2\rho^{2(p+1)})
\Bigl(4-(4p+3)\rho^{2p} + 4p \rho^{2(p+1)} - \rho^{2(2p+1)}\Bigr)
 \bigg].
\end{eqnarray}
Moreover, if $(V_n)$ has finite conditional moments of order $>4$, we also have
\begin{equation}
\label{CLTDW}
\sqrt{n}(\wh{D}_{n}-D)\liml
\cN\left(0,4\tau^2\right).
\end{equation}
\end{thm}
\noindent{\bf Proof.}
The proofs are given in Appendix\,B. \demend

\vspace{-2ex} 
\begin{rem}
{\rm
We now point out the crucial role played by the additional excitation in the control law given by \eqref{CONTROL}. 
It follows from
\eqref{EXPTAU} that if $\rho=0$, then $\tau^2$ reduces to
$$
\tau^2=\frac{\sigma^2 + \nu^2}{\nu^2}.
$$
Consequently, if $\nu^2=0$ i.e. there is no persistent excitation, then this variance explodes.
Therefore, the persistent excitation allows to investigate the important case $\rho=0$ and more
generally to stabilize the asymptotic variance of the Durbin-Watson statistic. 
} 
\end{rem}

We are now in the position to test whether or not the serial correlation parameter $\rho=0$.
According to Theorem \ref{T-ASCVGDW}, we have under the null hypothesis $\cH_0$,
$$
\lim_{n \rightarrow \infty} \wh{D}_n= 2
\hspace{1cm} \text{a.s.}
$$
In addition, we clearly have from \eqref{CLTDW} that under $\cH_0$,
\begin{equation}
\label{CHICLTDW}
\frac{n}{4\tau^2} \left( \wh{D}_{n} - 2 \right)^2 \liml \chi^2
\end{equation}
where $\chi^2$ stands for a Chi-square distribution with one degree of freedom.
It remains to accurately estimate the asymptotic variance $\tau^2$. It is not hard to see that
\begin{equation}
\label{NH}
\lim_{n \rightarrow \infty} \frac{1}{n}\sum_{k=0}^nX_k^2 =  \sigma^2+\nu^2
\hspace{1cm} \text{a.s.}
\end{equation}
Consequently, as $\nu^2$ is known, it immediately follows from \eqref{NH} that
\begin{equation*}
\wh{\sigma}^2_n = \frac{1}{n}\sum_{k=1}^n X_k^2 - \nu^2
\end{equation*}
converges almost surely to $\sigma^2$.
Hence, we can propose to make use of
\begin{eqnarray}
\label{TAUCHAP}
\wh{\tau}_n^2 &=& \frac{(1-\overline{\rho}^{\,2}_n)}
{(\wh{\sigma}^2_n+ \nu^2) (\nu^2+\wh{\sigma}^2_n\overline{\rho}_n^{\,2(p+1)}) }
\bigg[ 
\Bigl((\wh{\sigma}^2_n - \nu^2) - (p+1)\wh{\sigma}^2_n\overline{\rho}^{\,2p}_n + (p-1)
\wh{\sigma}^2_n\overline{\rho}^{\,2(p+1)}_n\Bigr)^2
\notag \\
& & \hspace{1cm} +\ \  
\wh{\sigma}^2_n(\nu^2+\wh{\sigma}^2_n\overline{\rho}^{\,2(p+1)}_n)
\Bigl(4-(4p+3)\overline{\rho}^{\,2p}_n + 4p \overline{\rho}^{\,2(p+1)}_n - \overline{\rho}^{\,2(2p+1)}_n\Bigr) 
\bigg].
\end{eqnarray}
Therefore, our bilateral statistical test relies on the following results.

\begin{cor}
\label{T-DWTESTCERO}
Assume that $(x_n)$ and $(V_n)$ have finite conditional moments of order $>4$. 
Then, under the null hypothesis $\cH_0\,:\,`` \rho = 0"$,
\begin{equation}
\label{DWTESTH0CERO}
T_n = \frac{n}{4\wh{\tau}_n^{\, 2}} \left( \wh{D}_{n} - 2 \right)^2 \liml \chi^2
\end{equation}
In addition, under the alternative hypothesis $\cH_1\,:\,`` \rho \neq 0"$,
\begin{equation}
\label{DWTESTH1CERO}
\lim_{n\rightarrow \infty}  T_n = + \infty \hspace{1cm} \textnormal{a.s.}
\end{equation}
\end{cor}
\vspace{1ex}

\noindent
From a practical point of view, for a significance level $\alpha$ where $0<\alpha <1$, 
the acceptance and rejection regions are given by 
$\cA= [0, a_{\alpha}]$ and $\mathcal{R} = ] a_{\alpha}, +\infty [$
where $a_{\alpha}$ stands for the $(1-\alpha)$-quantile of the Chi-square distribution with one degree of freedom.
The null hypothesis $\cH_0$ will be accepted if \ $T_n  \leq a_{\alpha}$,
and will be rejected otherwise. 
\vspace{1ex}

Let us now make a few comments.
First of all, under $\Hzero$, we already saw that $\tau^2$ reduces to 
$(\sigma^2 + \nu^2)/\nu^2$. It can be estimated by $(\wh{\sigma}^2_n + \nu^2)/\nu^2$.
Therefore, it is also possible to consider the test statistic associated with
\begin{equation}\label{Tntilde}
\cT_n = \frac{n^2 \nu^2}{4(\wh{\sigma}^2_n + \nu^2)} \left( \wh{D}_{n} - 2  \right)^2.
\end{equation}
Intuitively, one may think that the statistical test based on $\cT_n$
is more efficient under $\cH_0$ since we do not estimate the parameter
$\rho$, but less powerful under $\cH_1$. This point will be examined
in Section 4. Next, the acceptance of $\Hzero$ after our statistical test procedure should lead to 
a change of control law. As a matter of fact, if we accept $\rho=0$, 
the driven noise $(\veps_n)$ is not correlated. It means that we can implement 
the usual control law \cite{AW1995} associated with model (\ref{ARXP}) given, for all $n \geq 0$, by
\begin{equation*}
U_n  = x_{n+1} - \wh{\theta}_n^{\,t}\varphi_{n}
\end{equation*}
where $\wh{\theta}_n$ stands for the standard least-squares estimator associated with \eqref{ARXP}.
Finally, the test provided by Corollary \ref{T-DWTESTCERO} may be of course extended 
if we replace zero by any $\rho_0 \in \dR$ with $|\rho_0|<1$  in the null hypothesis. 
To be more precise, we are able as in \cite{BPV2012} to test 
$\Hzero:\,\og \rho = \rho_0\fg$ versus 
$\Hone:\,\og \rho \neq \rho_0\fg$.
We wish to mention that the asymptotic variance $\tau^2$ is smaller than the one obtained in  \cite{BPV2012}.

\section{Numerical Experiments}
This section is devoted to the application of our Durbin-Watson serial correlation test. 
Although this test has several potential of being applied in concrete situations, 
a large search in the literature did not offer any one. 
We then consider artificial models for illustrative purposes and for studying the empirical level and power of our test for sample sizes from small to moderate, that is $n=50$, 100, 200, 500, 1000 and 2000.
\vspace{1ex}

In order to keep this section brief, we restrict ourself to the three explosive models in open-loop
\begin{eqnarray}
\label{ARX1} 
X_{n+1} &=& \frac{3}{2}X_n + U_n + \varepsilon_{n+1} \\
\label{ARX2}
X_{n+1} &=& -X_n +2  X_{n-1} + U_n + \varepsilon_{n+1}\\
\label{ARX3}
X_{n+1} &=&  X_n + \frac{1}{2} X_{n-1} + \frac{1}{4} X_{n-2} + U_n + \varepsilon_{n+1}
\end{eqnarray}
where the driven noise $(\veps_n)$ is given by (\ref{COR}) and $(V_n)$ is a sequence of independent and identically distributed random
variables with $\cN(0,1)$ distribution. 
The control law $U_n$ is given by (\ref{CONTROL}) where, for the sake of simplicity, the reference trajectory $x_n = 0$
and the persistent excitation $(\xi_n)$ s a sequence of independent and identically distributed random
variables with $\cN(0,\nu^2)$ distribution.
\ \vspace{1ex} \par
For each model, we based our numerical simulations on $N=1000$ realizations of sample size $n$. 
We use a short learning period of $100$ time steps. This learning period allows us to forget the transitory phase.
The level of significance is set to $\alpha=5\%$. 
For the statistical tests based on $T_n$ and $\cT_n$,
we are interested in the empirical level under $\Hzero$ to be compared to
the theoretical level $5\%$, and the empirical power 
under $\Hone$, to be compared with 1.
\ \vspace{1ex} \par
First of all, let us study the effect of the variance $\nu^2$ of the exogenous noise $(\xi_n)$
on the behavior of the statistical test under $\Hzero$.

\begin{table}[H]
\begin{tabular}{|c||c|c|c|c|c|c|c|}
\hline 
 & $n$ & 50 & 100 & 200 & 500 & 1000 & 2000 \\
\hline 
\hline
$\nu = 0.5$  & $T_n$ & 0.9\% &1.6\% & 2\% & 3\% & 2.9\% & 4.9\% \\
\cline{2-8}
 & $\cT_n$ &  0\% & 0.1\% & 0.5\% & 2.1\% & 2.4\% & 4.4\%  \\
\hline
$\nu = 1$  & $T_n$ & 2.5\% & 2.5\% &  3.3\% & 4.4\% & 4.4\% & 4.9\% \\
\cline{2-8}
 & $\cT_n$ &  1.3\% & 1.3\% & 2.5\% & 4.1\% & 4\%  & 4.8\% \\
\hline 
$\nu = 2$   & $T_n$ &5.2\% & 4.1\% & 4.9\% & 5.3\% & 4.7\% & 4.6\% \\
\cline{2-8}
 & $\cT_n$ & 3.7\% & 3.7\% & 4.1\% & 5.1\% & 4.7\% & 4.6\% \\
\hline 
$\nu = 3$  & $T_n$ & 5.8\% & 5.1\% & 5.6\% & 4.3\% & 4.3\% & 4.8\% \\
\cline{2-8}
 & $\cT_n$ & 4.5\% & 4.6\% & 5.1\% & 4.2\% & 4.3\% & 4.7\% \\
\hline 
\end{tabular}
\vspace{1ex}
\caption{Model  (\ref{ARX1}). Percentage of rejections of our test under $\Hzero$ 
(to be compared to the $5\%$ theoretical level).}
\label{TABResNU}
\end{table}

It is clear from Table \ref{TABResNU}, where one can find the results 
obtained for different values of $\nu$, that the variance of the persistent excitation $(\xi_n)$
in the control law plays a crucial role.
Indeed, one can observe that if it is too small, then the empirical level 
of the test is bad for sample sizes from small to moderate  $n \leq 1000$.
Of course, a high value of $\nu^2$ improves the performance of 
the test under $\Hzero$, but degrades the performance of the tracking.
The value $\nu = 2$ realizes a good compromise and allows a good calibration of the test under $\Hzero$.

\begin{table}[H]
\begin{tabular}{|c||c|c|c|c|c|c|c|}
\hline 
 & $n$ & 50 & 100 & 200 & 500 & 1000 & 2000 \\
\hline 
Model (\ref{ARX1}) & $T_n$ &5.2\% & 4.1\% & 4.9\% & 5.3\% & 4.7\% & 4.6\% \\
\cline{2-8}
 & $\cT_n$ & 3.7\% & 3.7\% & 4.1\% & 5.1\% & 4.7\% & 4.6\% \\
\hline 
Model (\ref{ARX2}) & $T_n$ & 5.9\% & 3\% & 3.9\% & 4.6\% & 4.8\% &5.2\% \\
\cline{2-8}
 & $\cT_n$ & 4.7\% & 2.5\% & 3.7\% & 4.5\% & 4.8\% & 5.1\% \\
\hline 
Model (\ref{ARX3}) & $T_n$ & 4.8\% & 4.7\% & 4.1\% & 5.2\% & 4.9\% & 6\% \\
\cline{2-8}
 & $\cT_n$ & 3.8\% & 3.5\% & 3.9\% & 5\% & 4.9\% & 5.9\% \\
\hline 
\end{tabular}
\vspace{1ex}
\caption{Percentage of rejections of our test under $\Hzero$
(to be compared to the $5\%$ theoretical level). $\nu = 2$.}
\label{TABResHzero}
\end{table}

One can find in Table \ref{TABResHzero} the percentage of rejections of our test under $\Hzero$
for the three different models \eqref{ARX1} to \eqref{ARX3}. 
The empirical levels of the test are close to the $5\%$ theoretical level 
even for small sample sizes.
Both statistical tests based on $T_n$ and $\cT_n$ are comparable
even if the test statistic $\cT_n$ systematically tends to less reject $\Hzero$
than the test statistic $T_n$.
\vspace{1ex}

Let us now study the empirical power of our statistical test.
One can find in Tables \ref{TABResHoneARX1} to \ref{TABResHoneARX3}
the results obtained for each of the three models \eqref{ARX1} to \eqref{ARX3}. 
As expected, it is difficult to reject $\Hzero$ when $\rho = 0.05$ for small sample sizes
or $\rho = 0.1$ to a lesser extent. 
However, the test performs pretty well as the percentage of correct decisions increases with the sample size.

\begin{table}[H]
\begin{tabular}{|c||c|c|c|c|c|c|c|}
\hline 
& $n$ & 50 & 100 & 200 & 500 & 1000 & 2000 \\
\hline 
\hline
$\rho = 0.05$  & $T_n$  & 6.4\% & 7.1\% & 9.6\% & 18.8\% & 30.6\% & 50.3\% \\
\cline{2-8}
 & $\cT_n$ & 4.8\% & 6.2\% & 9\% & 18.4\% & 30.5\% & 50.2\% \\
\hline 
$\rho = 0.1$  & $T_n$  & 11.1\% & 17.4\% & 25.9\% & 56.8\% & 81.9\% & 98.2\% \\
\cline{2-8}
 & $\cT_n$ & 8.9\% & 15.2\% & 24.7\% & 56.2\% & 81.6\% & 98.1\% \\
\hline 
$\rho = 0.2$  & $T_n$  & 32\% & 47.5\% & 77.6\% & 98.9\% & 100\% & 100\%\\
\cline{2-8}
 & $\cT_n$ & 10.7\% & 24.1\% & 49.5\% &  91.2\% & 99.5\% & 100\%\\
\hline 
$\rho = 0.3$  & $T_n$  & 57\% & 83\% & 97.7\% & 100\% & 100\% & 100\%\\
\cline{2-8}
 & $\cT_n$ & 51.7\% & 81.7\% & 97.5\% & 100\% & 100\% & 100\%\\
\hline 
$\rho = 0.4$  & $T_n$  & 78.8\% & 97.2\% & 99.7\% & 100\%& 100\%& 100\%  \\
\cline{2-8}
 & $\cT_n$ & 73.4\% & 96.5\% & 99.7\% & 100\% & 100\% & 100\%\\
\hline 
\end{tabular}
\vspace{1ex}
\caption{Model (\ref{ARX1}). Percentage of correct decisions of our test under $\Hone$.}
\label{TABResHoneARX1}
\end{table}

We further observe, as expected, that for a fixed value of the sample size $n$, 
the higher the value of $\rho$ is, the more the percentage of correct decisions increases.
We also notice that for a fixed value of $\rho$, the empirical power increases 
with the sample size. In conclusion, the test performs very well under $\Hone$.
Moreover, higher values of the order $p$ does not degrade the performances of our statistical test.

\begin{table}[H]
\begin{tabular}{|c||c|c|c|c|c|c|c|}
\hline 
& $n$ & 50 & 100 & 200 & 500 & 1000 & 2000 \\
\hline 
\hline
$\rho = 0.05$  & $T_n$  & 5.2\% & 6.2\% & 11\% & 17.6\% & 28.5\% & 54.6\% \\
\cline{2-8}
 & $\cT_n$ & 4.5\% & 5.6\% & 10.7\% & 17.4\% & 28.3\% & 54.6\% \\
\hline 
$\rho = 0.1$  & $T_n$  & 11.5\% & 15\% & 25.3\% & 53\% & 78.1\% & 97.3\%\\
\cline{2-8}
 & $\cT_n$ & 9.7\% & 14.1\% & 24.6\% & 52.5\% & 78.1\% & 97.3\%\\
\hline 
$\rho = 0.2$  & $T_n$  & 28.5\% & 45.8\% & 70.9\% & 98.3\% & 99.9\% & 100\%\\
\cline{2-8}
 & $\cT_n$ & 24.4\% & 44.1\% & 69.6\% & 98.3\% & 99.9\% & 100\%\\
\hline 
$\rho = 0.3$  & $T_n$  & 53.1\% & 80.6\% & 97\% & 100\% & 100\% & 100\% \\
\cline{2-8}
 & $\cT_n$ & 48.4\% & 79.1\% & 96.9\% & 100\% & 100\% & 100\%\\
\hline 
$\rho = 0.4$  & $T_n$  & 77.4\% & 96.4\% & 100\% & 100\% & 100\% & 100\% \\
\cline{2-8}
 & $\cT_n$ & 74\% & 95.9\% & 100\% & 100\% & 100\% & 100\% \\
\hline 
\end{tabular}
\vspace{1ex}
\caption{Model (\ref{ARX2}). Percentage of correct decisions of our test under $\Hone$.}
\label{TABResHoneARX2}
\end{table}

Finally, one can realize that for small sample sizes, the statistical test 
based on $\cT_n$ is less powerful than the one associated with $T_n$.
We also wish to mention that, by symmetry, the performance of our statistical tests are 
the same for negative values of $\rho$.

\begin{table}[H]
\begin{tabular}{|c||c|c|c|c|c|c|c|}
\hline 
& $n$ & 50 & 100 & 200 & 500 & 1000 & 2000 \\
\hline 
\hline
$\rho = 0.05$  & $T_n$  &6.9\% & 7\% & 9.6\% &16\% &28.5\% & 50.6\% \\
\cline{2-8}
 & $\cT_n$ & 6.1\% & 6.6\% &9.4\% &15.9\% &28.5\% &50.5\% \\
\hline 
$\rho = 0.1$  & $T_n$  & 11\% & 14.2\% & 23.8\% & 53.6\% &83.5\% & 98\%\\
\cline{2-8}
 & $\cT_n$ & 9.6\% & 12.6\% & 23.6\% & 53\% &83.1\% &98\%\\
\hline 
$\rho = 0.2$  & $T_n$  & 29.7\% & 46.2\% &77.1\%  &97.7\% &100\% &100\% \\
\cline{2-8}
 & $\cT_n$ & 25.9\% &44.6\% &76.3\% &97.7\% &100\% &100\%\\
\hline 
$\rho = 0.3$  & $T_n$  & 52.5\% & 81.1\% & 97.4\% & 100\% &100\% &100\% \\
\cline{2-8}
 & $\cT_n$ & 48.8\% &79.2\% &97.3\% &100\% &100\% &100\%\\
\hline 
$\rho = 0.4$  & $T_n$  & 75.9\% & 96.9\% & 100\% & 100\% & 100\% & 100\% \\
\cline{2-8}
 & $\cT_n$ & 73.3\% & 96.6\% & 100\% & 100\% & 100\% & 100\% \\
\hline 
\end{tabular}
\vspace{1ex}
\caption{Model (\ref{ARX3}). Percentage of correct decisions of our test under $\Hone$.}
\label{TABResHoneARX3}
\end{table}


\section{Conclusion}

Thanks to the introduction of a persistent excitation in the control law
used to regulate an ARX(p,1) process, we were able to propose a non correlation 
test for the driven noise based on the well-known Durbin-Watson statistic.
In addition, we have shown through a simulation study on artificial models
the efficiency of our statistical test procedure. 
\ \vspace{1ex} \par
Of course, many questions remain open.
In particular, the extension of our results
to ARX$(p,q)$ processes where $q>1$, would be a very attractive challenge 
for the control community.
Even though our test is a potentially  useful tool, we have seen in the literature
that ARX$(p,1)$ models are often too simple 
for being applied to real physical models.
However, the study of such models is much more difficult to handle. 
On the one hand, it will be necessary
to make an additional assumption of strong controllability, see \cite{BVAUT2010}, \cite{BVIJC2010}. 
On the other hand, the asymptotic variance given by \eqref{EXPTAU} will be much
more complicated as well as its estimate given by \eqref{TAUCHAP}. 
\ \vspace{1ex} \par
Finally, this work may be seen as a first step towards a serial correlation test for 
ARX$(p,q)$ processes. The implementation of our statistical test within a real physical model
will allow to fully validate its efficiency. 

\section*{Appendix A}

\begin{center}
{\small PROOFS OF THE ALMOST SURE CONVERGENCE RESULTS}
\end{center}

\renewcommand{\thesection}{\Alph{section}} 
\renewcommand{\theequation}
{\thesection.\arabic{equation}} \setcounter{section}{1}  
\setcounter{equation}{0}
\ \\
The almost sure convergence results rely on the following keystone lemma.
\begin{lem} \label{L-CVGSN}
Assume that $(V_n)$ has a finite conditional moment of order $>2$. Then, we have
\begin{equation}
\label{CVGSN}
\lim_{n\rightarrow \infty} \frac{S_n}{n}=\Lambda \hspace{1cm} \text{a.s.}
\end{equation}
where $\Lambda$ is the symmetric square matrix of order $p+2$ given by 
\begin{equation}  
\label{DEFLAMBDA}
\Lambda=\left( 
\begin{array}{cc}
L & K^t \\ 
K &H
\end{array}
\right)
\end{equation}
with  $L=(\sigma^2 + \nu^2)\rI_{p+1}$ where $\rI_{p+1}$ stands for the identity matrix of order $p+1$,
$K$ is the line vector of $\dR^{p+1}$ 
\begin{equation}
\label{DEFK}
K=\Bigl(K_0,K_1,K_2,\ldots,K_p\Bigr)
\end{equation}
with $K_0=\nu^2$ and for  all $1\leq k \leq p$, $K_k=-(\sigma^2+\nu^2)\theta_k-\sigma^2\rho^k$ and 
$H$ is the positive real number given by
\begin{equation}
\label{DEFH}
H=\nu^2+\sigma^2\sum_{k=1}^p (\theta_k+\rho^k)^2+\nu^2\sum_{k=1}^p \theta_k^2+\frac{\sigma^2\rho^{2(p+1)}}{1-\rho^2}.
\end{equation}
\end{lem}

The proof of Lemma \ref{L-CVGSN} is left to the reader as it follows exactly the same lines as the one of Theorem 4.1 in \cite{BVIJC2010}.
Denote by $S$ the Schur complement of  $L$ in $\Lambda$,
$$
S= H - \frac{1}{\sigma^2 + \nu^2} \parallel\! K \!\parallel^2= \frac{\sigma^2(\nu^2+\sigma^2\rho^{2(p+1)})}{(1-\rho^2)(\sigma^2+\nu^2)}.
$$
We deduce from \eqref{DEFLAMBDA} that
$$
\det(\Lambda) = S\det(L) = S (\sigma^2+\nu^2)^{p+1}=\frac{\sigma^2(\nu^2+\sigma^2\rho^{2(p+1)})(\sigma^2+\nu^2)^p}{1-\rho^2}.
$$
Consequently, whatever the value of the parameter $\rho$ with $|\rho|<1$, $\det(\Lambda) \neq 0$ which means that the matrix 
$\Lambda$ is always invertible. The almost sure convergence of the least squares estimator $\wh{\vartheta}_{n}$ of the parameter 
$\vartheta$ associated with the $\text{ARX}(p+1,2)$ process given by \eqref{EQFONDA}
is as follows.

\begin{thm} 
\label{T-ASCVGVARTHETA}
Assume that $(V_n)$ has a finite conditional moment of order $>2$. 
Then, $\wh{\vartheta}_{n}$ converges almost surely to $\vartheta$,
\begin{equation}  
\label{ASCVGVARTHETA}
\parallel \wh{\vartheta}_{n}-\vartheta \parallel^{2}= \mathcal{O} 
\left( \frac{\log n}{n} \right) 
\hspace{1cm}\text{a.s.}
\end{equation}
\end{thm}

\begin{proof}
We deduce from \eqref{LSVARTHETA} and \eqref{DEFSN} that
\begin{equation}
\label{DEVARTHETA}
\wh{\vartheta}_{n}-\vartheta =S_{n-1}^{-1}\Bigl( M_n + \wh{\vartheta}_{0}-\vartheta\Bigr)
\end{equation}
where 
$$
M_n=\sum_{k=1}^n \Phi_{k-1}V_{k}.
$$
The sequence $(M_n)$ is a locally square-integrable $(p+2)$-dimensional martingale with increasing process
$$
\cM_n = \sigma^2 \sum_{k=0}^{n-1} \Phi_k \Phi_k^t.
$$
Then, it follows from the strong law of large numbers for martingales given e.g. in Theorem 4.3.16 of \cite{DUF1997} 
that
\begin{equation}
\label{SLLNVARTHETA}
\parallel \wh{\vartheta}_{n+1}-\vartheta \parallel ^{2}=\mathcal{O}\left( 
\frac{\log \lambda _{max}(S_{n})}{\lambda _{min}(S_{n})}\right) \hspace{1cm}
\text{a.s.}  
\end{equation}
Therefore, we clearly obtain \eqref{ASCVGVARTHETA} from \eqref{CVGSN} and \eqref{SLLNVARTHETA}.
\end{proof}

We immediately deduce from Theorem \ref{T-ASCVGVARTHETA} the almost sure convergence of
the least squares estimators $\wh{\theta}_{n}$ and $\wh{\rho}_{n}$ to $\theta$ and $\rho$.

\begin{cor}
\label{C-ASCVGTHETARHO}
Assume that $(V_n)$ has a finite conditional moment of order $>2$. 
Then, $\wh{\theta}_{n}$ and $\wh{\rho}_{n}$ both converge almost surely to $\theta$ and $\rho$,
\begin{equation}  
\label{ASCVGTHETA}
\parallel \wh{\theta}_{n}-\theta \parallel^{2}= \mathcal{O} 
\left( \frac{\log n}{n} \right) 
\hspace{1cm}\text{a.s.}
\end{equation}
\begin{equation}  
\label{ASCVGRHO}
 (\wh{\rho}_{n}-\rho)^{2}= \mathcal{O} 
\left( \frac{\log n}{n} \right) 
\hspace{1cm}\text{a.s.}
\end{equation}
\end{cor}

\noindent
{\bf Proof of Theorem \ref{T-ASCVGDW}.} The proof of Theorem \ref{T-ASCVGDW} relies on
Corollary \ref{C-ASCVGTHETARHO}. It is left to the reader inasmuch as it follows essentially the same lines as
those in Appendix C of \cite{BPV2012}.
\hfill
$\mathbin{\vbox{\hrule\hbox{\vrule height1ex \kern.5em\vrule height1ex}\hrule}}$

\section*{Appendix B}

\begin{center}
{\small PROOFS OF THE ASYMPTOTIC NORMALITY RESULTS}
\end{center}

\renewcommand{\thesection}{\Alph{section}} 
\renewcommand{\theequation}
{\thesection.\arabic{equation}} \setcounter{section}{2}  
\setcounter{equation}{0}


We shall now prove Theorem \ref{T-CLTDW}.
First of all, we obtain from \eqref{NEWLSRHO} that
\begin{equation}
\label{DEVRHOVER}
\overline{\rho}_n=\frac{I_n}{J_{n-1}}
\end{equation}
where
\begin{equation*} 
I_{n} =\sum_{k=1}^{n}\wh{\veps}_{k} \wh{\veps}_{k-1}
\hspace{1cm}\text{and}\hspace{1cm}
J_{n} =\sum_{k=0}^{n}\wh{\veps}_{k}^{\,  2}.
\end{equation*}
As in \cite{BPV2012}, we deduce from \eqref{DEVARTHETA} and \eqref{DEVRHOVER} the martingale decomposition
\begin{equation}  
\label{MAINDECOCLTTHETARHO}
\sqrt{n} \begin{pmatrix}
\ \wh{\vartheta}_{n} - \vartheta  \\
\ \overline{\rho}_{n} - \rho
\end{pmatrix}
 =\frac{1}{\sqrt{n}} \cA_n Z_n + \cB_n
\end{equation}
where $(Z_n)$ is the locally square-integrable $(p+3)$-dimensional martingale given by
\begin{equation*}
Z_n= 
\begin{pmatrix}
M_n  \\
N_n
\end{pmatrix}
\end{equation*}
with
$$
M_n = \sum_{k=1}^n \Phi_{k-1}V_{k}
\hspace{1cm}\text{and} \hspace{1cm}
N_n = \sum_{k=1}^{n}\veps_{k-1}V_{k}.
$$
In addition, it follows from Lemma \ref{L-CVGSN} that the sequences
$(\cA_n)$ and $(\cB_n)$ converge almost surely to $\cA$ and $\cB$ given by
$$
\cA= 
 \begin{pmatrix}
\Lambda^{-1} & \ \ 0_{p+2}  \vspace{1ex}\\
\sigma^{-2}(1- \rho^2) \cC^{t} & \sigma^{-2}(1- \rho^2)
\end{pmatrix},
\hspace{1cm}
\cB= \begin{pmatrix}
\ 0_{p+2} \ \vspace{1ex}\\
 0
\end{pmatrix}
$$
where $0_{p+2}$ stands for the null vector of $\dR^{p+2}$ and $\Lambda$ is the matrix given by \eqref{DEFLAMBDA}.
Moreover, the vector $\cC$ belongs to $\dR^{p+2}$ with
$$\cC=(1-\rho^2)  \Lambda^{-1} \nabla^{t} \rJ_p ^{t} T$$ 
where $\rJ_p=(\rI_p \ \ 0_p)$, $T$ is the vector of $\dR^{p}$ given by
$T^{t}=(1, \rho, \ldots,  \rho^{p-1} )$ and $\nabla$ is the rectangular matrix of size $(p+1)\!\times\!(p+2)$ given by
\begin{equation*}
\nabla=
\begin{pmatrix}
1 & 0 & \cdots & \cdots & \cdots & 0 & 1 \\ 
\rho & 1 & 0 & \cdots & \cdots & 0 & \rho - \xi_1 \\ 
\rho^2 & \rho & 1 & 0 & \cdots & 0 & \rho^2 -  \xi_2\\ 
\cdots & \cdots & \cdots & \cdots & \cdots & \cdots & \cdots \\ 
\rho^{p-1} & \rho^{p-2} & \cdots & \rho & 1 & 0 & \rho^{p-1} -\xi_{p-1}\\ 
0 & 0 & \cdots & \cdots & \cdots & 0 & -1
\end{pmatrix}
\end{equation*}
where, for all $1\leq k \leq p-1$, $\xi_k$ is the weighted sum
$$
\xi_k=\sum_{i=1}^{k}\rho^{k-i}\theta_i.
$$
We already saw that $(Z_n)$ is a martingale with predictable quadratic variation given, for all $n \geq 1$, by
\begin{equation*}
\langle Z \rangle_n=
\sigma^2 \sum_{k=0}^{n-1}
\begin{pmatrix}
\ \Phi_k \Phi_k^{t} \ & \ \Phi_k \veps_k \ \\
\ \Phi_k^{t}\veps_k \ & \  \veps_k^2 \ 
\end{pmatrix}.
\end{equation*}
Hence, we deduce once again from Lemma \ref{L-CVGSN} that
\begin{equation*}
\lim_{n\rightarrow \infty} \frac{1}{n} \langle Z \rangle_n= \cZ
\hspace{1cm}\text{a.s.}
\end{equation*}
where $\cZ$ is the positive-definite symmetric matrix given by
\begin{equation*}
\cZ=\sigma^4
\begin{pmatrix}
\sigma^{-2}\Lambda & \zeta  \vspace{1ex}\\
\zeta^{t}  & \eta
\end{pmatrix}
\end{equation*}
where $\zeta$ is the vector of $\dR^{p+2}$ such that
$\zeta^{t}=(1, \rho, \ldots, \rho^p, \varrho_p)$ with
$$
\varrho_p= - \eta \rho^2 - \sum_{i=1}^p  \rho^{i} \theta_{i}
\hspace{1cm} \text{and} \hspace{1cm} 
\eta = \frac{1}{1- \rho^2}.
$$
As $(Z_n)$ satisfies the Lindeberg condition, we deduce from the central limit theorem for multidimensional martingales 
given e.g. by Corollary 2.1.10 in \cite{DUF1997} that
\begin{equation*}
\frac{1}{\sqrt{n}} Z_{n} \liml \cN \Bigl(0, \cZ \Bigr)
\end{equation*}
which, via the martingale decomposition \eqref{MAINDECOCLTTHETARHO}
and Slutsky's lemma, leads to
\begin{equation}  
\label{MAINCLT}
\sqrt{n} \begin{pmatrix}
\ \wh{\vartheta}_{n} - \vartheta  \\
\ \overline{\rho}_{n} - \rho
\end{pmatrix}
 \liml  \cN \Bigl(0, \cA \cZ \cA^{\prime}\Bigr).
\end{equation}
Therefore, we immediately obtain from \eqref{MAINCLT} that
\begin{equation}
\label{CLTPRHO}
\sqrt{n}(\overline{\rho}_{n}-\rho)\liml
\cN\left(0,\tau^2\right)
\end{equation}
where the asymptotic variance $\tau^2$ is given by
$
\tau^2 =  (1-\rho^2)^{2} ( \sigma^{-2}\cC^{t} \Lambda \cC + 2 \cC^{t} \zeta + \eta ).
$
It follows from tedious but straighforward calculations that $\tau^2$ coincides with the
expansion given by \eqref{EXPTAU}. Finally, as
\begin{equation}
\label{DECOCLTDW}
\wh{D}_n - D = -2 (\overline{\rho}_n - \rho) + R_n
\end{equation}
where the remainder $R_n$ is negligeable which means that 
\begin{equation*}
R_n =  o\left(\frac{1}{\sqrt{n}} \right) 
\hspace{1cm}\text{a.s.}
\end{equation*}
we obtain \eqref{CLTDW} from \eqref{CLTPRHO} and \eqref{DECOCLTDW}, which
achieves the proof of  Theorem \ref{T-CLTDW}.
\hfill
$\mathbin{\vbox{\hrule\hbox{\vrule height1ex \kern.5em\vrule height1ex}\hrule}}$

\ \vspace{-1ex} \\
{\bf Acknowledgements.}
The authors would like to thanks the anonymous reviewers for their
constructive comments which helped to improve the paper substantially.

\end{document}